\documentclass[12pt]{amsart}
\usepackage{amsmath}
\usepackage{amsfonts}
\usepackage{amsthm}
\usepackage{amssymb}
\usepackage{amscd}
\usepackage[all]{xy}
\usepackage{enumerate}
\usepackage{tikz}
\usetikzlibrary{positioning}
\keywords{ Grassmannian, periods domain, infinitesimal deformation, adjoint forms, spectral sequence} 

\subjclass{14C34, 14D07, 14J10, 14J40, 14J70, 14M15.}

\theoremstyle{plain}
\newtheorem{thm}{Theorem}[subsection]

\newtheorem{prop}[thm]{Proposition}

\newtheorem{lem}[thm]{Lemma}

\theoremstyle{definition}
\newtheorem{defn}[thm]{Definition}

\newtheorem{rmk}[thm]{Remark}

\newcommand{\sB}{\mathcal{B}}

\newcommand{\sE}{\mathcal{E}}
\newcommand{\sF}{\mathcal{F}}
\newcommand{\sG}{\mathcal{G}}

\newcommand{\sI}{\mathcal{I}}
\newcommand{\sJ}{\mathcal{J}}

\newcommand{\sL}{\mathcal{L}}

\newcommand{\sO}{\mathcal{O}}
\newcommand{\sP}{\mathcal{P}}

\newcommand{\sR}{\mathcal{R}}

\newcommand{\sX}{\mathcal{X}}

\newcommand{\mC}{\mathbb{C}}

\newcommand{\mG}{\mathbb{G}}

\newcommand{\mP}{\mathbb{P}}

\newcommand{\Ima}{\mathrm{Im}\,}
\newcommand{\Ker}{\mathrm{Ker}\,}

\numberwithin{equation}{section}

\newcommand{\beba}  {\begin{equation}\begin{array}{rcl}}

\newcommand{\eaee}  {\end{array}\end{equation}}

\let\oldtocsection=\tocsection

\let\oldtocsubsection=\tocsubsection

\let\oldtocsubsubsection=\tocsubsubsection

\renewcommand{\tocsection}[2]{\hspace{0em}\oldtocsection{#1}{#2}}
\renewcommand{\tocsubsection}[2]{\hspace{1em}\oldtocsubsection{#1}{#2}}
\renewcommand{\tocsubsubsection}[2]{\hspace{2em}\oldtocsubsubsection{#1}{#2}}

\title{On Green's proof of infinitesimal Torelli theorem for hypersurfaces}

\author{Luca Rizzi}
\address{D.I.M.I. \\
the University of Udine\\
Udine, 33100 Italy\\
\texttt{rizzi.luca@spes.uniud.it}}

\author{Francesco Zucconi}

\address{D.I.M.I. \\
the University of Udine\\
Udine, 33100 Italy\\
\texttt{Francesco.Zucconi@dimi.uniud.it}}

\begin{document}

\markboth{Rizzi and Zucconi}{A Torelli Theorem for Grassmannian hypersurfaces}

\begin{abstract} We prove an equivalence between the infinitesimal Torelli theorem for top forms on a hypersurface contained inside a Grassmannian $\mathbb G$ and the theory of adjoint volume forms presented in \cite{RZ2}. More precisely,
via this theory and a suitable generalization of Macaulay's theorem we show that the differential of the period map vanishes on an infinitesimal deformation if and only if certain 
explicitly given twisted  volume forms go in the generalized Jacobi ideal of $X$ via the cup product homomorphism.
 \end{abstract}

\maketitle
\tableofcontents

\section{Introduction} Let $\sL$ be a line bundle over a smooth variety $Y$ and $X\subset Y$ the zero locus of  a global section $\sigma\colon Y\to\sL$. 
To study the infinitesimal deformations of $X$, M. L. Green introduced the notion of {\it{pseudo-Jacobi ideal}} $\sJ_{\sL,\sigma}$ \cite[Formula (2.11) page 144] {green1}.
 In particular the quotient $R_{L,\sigma}:= H^{0}(X,\sL)/\sJ_{\sL,\sigma}$ coincides with the tangent space of the Kuranishi family of (embedded) deformations and it is a piece of a graded ring 
 $\sR:=\oplus_{m\geq 0}R_{\sL^{\otimes m},\sigma}$, the so called {\it{generalized Jacobi ring}}; a notion that gives back the standard Jacobian ring if $Y$ is a projective space and $X=(F=0)$ is a smooth hypersurface. Following the fundamental papers by Griffiths, see \cite{Gri1}, in \cite{green1} it is proved that if $X$ is sufficiently ample the infinitesimally Torelli theorem holds for $X$, that is $d\sP_{X}$ is injective, and it is also possible to associate to the couple $(Y,X)$ a multiplicative structure on $\sR$ which is a perfect pairing.
  
 Now suppose that ${\rm{Pic}}(Y)=[H]\cdot \mathbb Z$ where $H$ is an effective divisor. Hence for any effective divisor $X\subset Y$ there exists a unique $m\in\mathbb N_{\geq 0}$ such that $X$ is an element of the linear system $|mH|$. Thus a natural problem is to find the minimum $m\in\mathbb N$ such that the multiplicative structure on $\sR$ gives a perfect pairing and $d\sP_{X}$ is injective where $X$ is a smooth element of $|mH|$; as far as we know the problem has been fully solved only in the case where $Y$ is a projective space, see also: \cite{Do} and  c.f. \cite{RZ2}, and in the case of K\"ahler C-spaces; see: \cite{K}. In this work we consider the case of Grassmannians, which is a particular case of \cite{K}, but we give a criterion to check if a local family is trivial, in terms of the geometry of certain top forms; the reader will realize that the ampleness degree in Green's proof can be clarified for many other important ambient spaces using our method. 

Let $Y=\mathbb G:= G(s,l+1)$ be the Grassmanniann of $s$-planes in $\mathbb C^{l+1}$ where $1<s<l$ and $l\geq 3$. Let $X\subset\mathbb G$ be an effective divisor. By Lefschetz theorem we know that if $\sO_{\mathbb G}(1)$ is the invertible sheaf which gives the Pl\"ukher embedding and $H$ is an hyperplane section, then $X$ is the zero locus of a global section $\sigma$ of $|aH|$ where $a\in \mathbb N_{>0}$. As a consequence of \cite{K} we know that, if $a\geq 3$, $d\sP_{X}$ is injective. In this paper we only consider the case where $X$ is of general type or of Calabi Yau type. Once we know that the infinitesimal Torelli theorem holds, a basic problem stands out if we consider the embedding $X\subset\mathbb G$. Indeed if $X=(\sigma=0)$, $\sigma\in H^{0}(\mathbb G,\sO_{\mathbb G}(a))$, any infinitesimal deformation is induced by a local family $(\sigma+\epsilon\tau=0)$ where $\tau\in H^{0}(\mathbb G,\sO_{\mathbb G}(a))$; see: Proposition \ref{etutto}. Hence it would be useful to have criteria to check which of these families actually induce the trivial deformation on $X$. The theory of generalized adjoint forms is a tool to solve this problem.  

Let us briefly recall the notion of generalized adjoint forms. The details of the general theory are discussed in \cite{RZ2}; here we recall that this theory has been successfully used in \cite{BAN}, \cite{Ra}, 
\cite{CNP}, \cite{G}, \cite{RZ3}, and that the foundations of the theory  of adjoint forms in dimension $\geq 2$ are in \cite{PZ}; see also \cite{RZ1} and \cite{R}.

In our case we twist by $\sO_X(2)$ the exact sequence associated to the infinitesimal deformation $\xi\in H^1(X,\Theta_X)$ to obtain
\begin{equation}
\label{estensionedivisoreduedue}
0\to\sO_X(2)\to \Omega^1_{\sX|X}(2)\to\Omega^1_X(2)\to0.
\end{equation}
The cup-product homomorphism $\partial_{\xi}\colon H^0(X,\Omega^1_X(2))\to H^1(X,\sO_X(2))$
 is trivial since $H^1(X,\sO_X(2))=0$. Now take a generic $n+1$-dimensional vector space $W<H^0(X,\Omega^1_X(2))$ and denote by $\lambda^{i}W$ the image of $\bigwedge^iW$ through the natural homomorphism $\lambda^i\colon\bigwedge^i H^0(X,\Omega^1_X(2))\to  H^0(X,\bigwedge^{i}(\Omega^1_X(2)))$.
Let $\sB:=\langle\eta_{1},\ldots,\eta_{n+1}\rangle$ be a basis of $W$ and $s_1,\ldots, s_{n+1}\in H^0(X, \Omega^1_{\sX|X}(2))$ liftings of, respectively,  $\eta_{1},\ldots,\eta_{n+1}$, then the map 
$$\Lambda^{n+1}\colon\bigwedge^{n+1} H^0(X, \Omega^1_{\sX|X}(2))\to H^0(X,  {\rm{det}}(\Omega^1_{\sX|X}(2)))$$ gives the twisted volume form $ \Omega:=\Lambda^{n+1}(s_1\wedge s_2 \wedge\ldots\wedge s_{n+1})\in H^0(X,{\rm{det}}( \Omega^1_{\sX|X}(2)))$ which is called {\it{generalized adjoint form associated to}} $\xi$, $W$, and $\sB$. On the other hand we also have $n+1$ top forms of $\Omega^1_X(2)$. Indeed consider the $n+1$ elements $\omega_i:=\lambda^{n}(\eta_1\wedge\ldots\wedge\eta_{i-1}\wedge{\widehat{\eta_i}}\wedge\eta_{i+1}\wedge\ldots\wedge\eta_{n+1})$, $i=1,\ldots,n+1$ obtained by the basis 
 $\langle \eta_1\wedge\ldots\wedge\eta_{i-1}\wedge{\widehat{\eta_i}}\wedge\eta_{i+1}\wedge\ldots\wedge\eta_{n+1}\rangle_{i=1}^{n+1}$ of $\bigwedge^n W$. By the the sequence (\ref{estensionedivisoreduedue}), 
 ${\rm{det}}(\Omega^1_{\sX|X}(2))={\rm{det}}(\Omega^1_{X}(2))\otimes_{\sO_{X}}\sO_{X}(2)$, and we can construct an obvious homomorphism:
\begin{equation}
\label{aggiuntasequenzaaa}
H^0(X,\sO_X(2))\otimes \lambda^nW \to H^0(X, {\rm{det}}(\Omega^1_{\sX|X}(2))\end{equation} 

The Generalized Adjoint Theorem, see \ref{theoremA} (and the therein quoted bibliography), fully characterizes  the condition $\Omega \in \Ima(H^0(X,\sO_X(2)\otimes \lambda^nW \to H^0(X,{\rm{det}}(\Omega^1_{\sX|X}(2))$; but the important point is that to check this condition is tantamount to check if $d\sP_{X}(\xi)=0$. More deeply, we can construct an explicit space of generalized adjoint forms associated to $\xi$ in the following way.
First we lift the sections $\eta_i$'s from $H^0(X,\Omega^1_X(2))$ to $H^0(\mathbb G,\Omega^1_\mathbb G(2))$, then we take the wedge product to obtain a twisted volume form $\widetilde{\Omega}$, independent from $\xi$, such that 
$\widetilde{\Omega}\in H^0(\mG,\Omega_\mG^N(2N))$. Finally by taking the form $R\in H^0(\mG,\sO_\mG(a))$ which induces the deformation $\xi$, we consider the class $R\cdot{\widetilde{\Omega}}\in  H^0(\mG,\Omega_\mG^N(2N+a))$ which by adjunction restricts to $\Omega\in H^0(X,\Omega^{N-1}_X(2N))=H^0(X,{\rm{det}}(\Omega^1_{X}(2))\otimes_{\sO_{X}}\sO_{X}(2))$. We show:
\medskip

\noindent
{\bf{Main Theorem}}
{\it{Let $a>l$ and let $(\sigma=0)=X\in |aH|$ as above. The following are equivalent:
\begin{itemize}
	\item[\textit{i})] the differential of the period map $d\sP_X$ is zero on the infinitesimal deformation $\xi$ induced by $R\in H^0(\mG,\sO_\mG(a))$;
	\item[\textit{ii})] $R$ is an element of the pseudo-Jacobi ideal $\sJ_{\sO_\mG(a),\sigma}$;
	\item[\textit{iii})]  for a generic $\xi$-adjoint $\Omega$ it holds $\Omega\in \Ima H^0(X,\sO_X(2))\otimes \lambda^nW\to H^0(X,\Omega^{N-1}_X(2N))$;
	\item[\textit{iv})] if $\widetilde{\Omega}\in H^0(\mG,\Omega_\mG^N(2N))$ restricts to a generalized adjoint form then $R\widetilde{\Omega}\in \sJ_{\Omega_\mG^N(2N+a),\sigma}$.
\end{itemize}
}}
\medskip

This last part of our theory relies, via our suitable generalization of the Macaulay's theorem (see: Theorem \ref {macaulaygen} and Theorem \ref{macaulaygenc} below), on a perfect equivalence between the infinitesimal theory of periods and our theory of adjoint forms 
which will be explored in the more general context of rational homogeneous varieties in a forthcoming paper.

\section[Sufficiently ample hypersurfaces]{Review on the theory of adjoint forms}
\label{sezione0}
\subsection{Definition of generalized adjoint form} The theory of generalized adjoint is built in \cite{RZ2}. Here we recall only the basic notions we need.

Let $X$ and $\xi$ be respectively a smooth compact complex variety of dimension $m$ and a class $\xi\in \text{Ext}^1(\sF,\sL)$ where 
$\sF$ and $\sL$ are two locally free sheaves on $X$ of rank $n$ and $1$ respectively. Then the extension class $\xi$ gives a rank $n+1$ vector bundle $\sE$ on $X$ which fits in an exact sequence:
\begin{equation}
\label{sequenza}
0\to\sL\to \sE\to\sF\to 0
\end{equation}

By wedge-sequences naturally associated to the sequence (\ref{sequenza}) we find that the invertible sheaf $\det\sF:=\bigwedge^n\sF$ fits into the exact sequence:
\begin{equation}
\label{wedge}
0\to\bigwedge^{n-1}\sF\otimes\sL\to \bigwedge^n\sE\to \det\sF\to 0, 
\end{equation} which still corresponds to $\xi$ under the isomorphism $\text{Ext}^1(\sF,\sL)\cong\text{Ext}^1(\det\sF,\bigwedge^{n-1}\sF\otimes\sL)\cong H^1(X,\sF^\vee\otimes\sL)$.

 A natural problem is to find conditions on the behavior of the global sections of the involved vector bundles in order to have the splitting of (\ref{sequenza}). From now on, assume that the connecting homomorphism
 $\partial_\xi \colon H^0(X,\sF)\to H^1(X,\sL)$ has a kernel of dimension sufficiently hight. More precisely assume that there exists a subset $W\subset \ker(\partial_\xi)$ of dimension $n+1$. Choose a basis $\mathcal{B}:=\{\eta_1,\ldots,\eta_{n+1}\}$ of $W$. By definition we can take liftings $s_1,\ldots,s_{n+1}\in H^0(X,\sE)$ of the sections $\eta_1,\ldots,\eta_{n+1}$. If we consider the natural map
\begin{equation*} 
\Lambda^n\colon \bigwedge^{n}H^0(X,\sE)\to H^0(X,\bigwedge^n\sE)
\end{equation*} we can define the sections
\begin{equation}
\label{Omegai}
\Omega_i:=\Lambda^n(s_1\wedge\ldots\wedge\hat{s_i}\wedge\ldots\wedge s_{n+1})
\end{equation} for $i=1,\ldots,n+1$. Denote by $\omega_i$, for $i=1,\ldots,n+1$, the corresponding sections in $H^0(X,\det\sF)$. By commutativity between evaluation of wedge product and restriction it easily follows that $\omega_i=\lambda^n(\eta_1\wedge\ldots\wedge\hat{\eta_i}\wedge\ldots\wedge \eta_{n+1})$, where $\lambda^n$ is the natural morphism 
$$\lambda^n\colon \bigwedge^{n}H^0(X,\sF)\to H^0(X,\det\sF).$$
\begin{defn} We denote by $\lambda^nW$ the vector subspace of $H^0(X,\det\sF)$ generated by $\omega_1,\ldots,\omega_{n+1}$.
If $\lambda^nW$ is nontrivial, the induced sublinear system $|\lambda^n W|\subset \mP(H^0(X,\det\sF))$ is called \emph{adjoint sublinear system of $W$}. We call $D_W$ {\it{its fixed divisor}} 
and $Z_W$ {\it{the base locus of its moving part}} $|M_W|\subset\mP(H^0(X,\det\mathcal{F}(-D_W)))$.
\end{defn}

\begin{defn} The form $\Omega\in H^0(X,\det\sE)$ corresponding to $s_1\wedge\ldots\wedge s_{n+1}$ via 
\begin{equation}
\Lambda^{n+1}\colon \bigwedge^{n+1}H^0(X,\sE)\to H^0(X,\det\sE)
\end{equation} is called {\it{ a generalized adjoint form of W}}.
\end{defn} 
\begin{rmk}
\label{zeri}
It is easy to see by local computation that this section is in the image of the natural injection $\det\sE(-D_W)\otimes\sI_{Z_W}\to \det\sE$.
\end{rmk}

The basic idea of the theory of adjoint forms is that, in a split exact sequence, generalized adjoint forms as $\Omega$ do not add any information which is not already given by the top forms $\omega_i\in H^{0}(X, \det\sF) $, $i=1,\dots,n+1$. More precisely, since ${\rm{det}}(\sE)={\rm{det}}(\sF)\otimes\sL$, everything is reduced to check the condition

\begin{equation}
\label{aggiuntazero2}
\Omega \in \Ima(H^0(X,\sL)\otimes \lambda^nW \to H^0(X,\det\sE)).
\end{equation}   By \cite[{Theorem [A]} and {Theorem [B]}]{RZ2} we have:
\begin{thm}\label{theoremA}
Let $\Omega\in H^0(X,\det\sE)$ be a generalized adjoint form associated to $W$ as above.
If $\Omega \in \Ima(H^0(X,\sL)\otimes \lambda^nW\to H^0(X,\det\sE))$ 
then $\xi\in\ker(H^1(X,\sF^\vee\otimes\sL)\to H^1(X,\sF^\vee\otimes\sL(D_W)))$. Viceversa. If $H^0(X,\sL)\cong H^0(X,\sL(D_W))$ and if 
 $\xi\in\ker(H^1(X,\sF^\vee\otimes\sL)\to H^1(X,\sF^\vee\otimes\sL(D_W)))$, then $\Omega \in \Ima(H^0(X,\sL)\otimes \lambda^nW\to H^0(X,\det\sE))$.
\end{thm}
\begin{rmk}\label{zerozero} Note that if $D_W=0$ the above theorem is a criterion for the vanishing of $\xi$.
\end{rmk}

\section[Sufficiently ample hypersurfaces]{The pseudo-Jacobi ring}
\label{sezione1}
Recall the following definition from \cite{green1}
\begin{defn}
We say that a property holds for a sufficiently ample line bundle $\sL$ on a projective variety $X$ if there exists an ample line bundle $\sL_0$ such that the property holds for all $\sL$ with $\sL\otimes \sL_0^{-1}$ ample.
\end{defn}

Take an $n$-dimensional smooth variety $Y$ and a sufficiently ample line bundle $L$ on $Y$. Let $\sigma\in H^0(Y,L)$ be a global section and $X$ the corresponding divisor. Assume that $X$ is smooth.

\subsection{Pseudo-Jacobi Ideal}

In the case that we are studying, the usual Jacobian ideal can be replaced by the so called pseudo-Jacobi ideal introduced in \cite{green1} and \cite{green2}. We briefly recall how it is constructed.

Given a line bundle $L$ on $Y$ and the sheaf $\Theta_Y$ of regular vector fields, consider the extension
\begin{equation}
0\to \sO_Y\to \Sigma_L\stackrel{\tau}{\rightarrow} \Theta_Y\to 0
\label{prolongation}
\end{equation} with extension class $-c_1(L)\in H^1(Y,\Omega^1_Y)$. $\Sigma_L$ is a sheaf of differential operators of order less or equal to $1$ on the sections of $L$.
In an open subset of $Y$ with coordinates $x_1,\ldots,x_n$ this sheaf is free and is generated by the constant section $1$ and the sections $D_i$, for $i=1,\dots,n$, which operates on the sections of $L$ by
\begin{equation*}
D_i(f\cdot l)=\frac{\partial f}{\partial x_i}\cdot l
\end{equation*}  where $l$ is a trivialization of $L$. The operators $D_i$ are sent to $\frac{\partial}{\partial x_i}$ in $\Theta_Y$.

 In particular to a global section $\sigma$ of $L$, we can associate a global section $\widetilde{d\sigma}$ of $L\otimes\Sigma_L^\vee$. If locally $\sigma=f\cdot l$, then $\widetilde{d\sigma}$ is given by
\begin{equation}\label{diffs}
\widetilde{d\sigma}=f\cdot l\cdot 1^\vee+\sum_{i=1}^n \frac{\partial f}{\partial x_i}\cdot l\cdot D_i^\vee
\end{equation} where $\{1^\vee,D_1^\vee,\ldots,D_n^\vee\}$ is a local basis of $\Sigma_L^\vee$ dual to $\{1,D_1,\ldots,D_n\}$.

Given a line bundle $E$, the contraction by $\widetilde{d\sigma}$ gives a map
\begin{equation*}
E\otimes \Sigma_L\otimes L^\vee\to E.
\end{equation*} 
To give an idea in the case $E=\sO_X$, the contraction $\Sigma_L\otimes L^\vee\to \sO_Y$ is given explicitly in local coordinates by 
\begin{equation*}
a_0\cdot l^\vee\otimes 1+\sum_{i=1}^n a_i \cdot l^\vee\otimes D_i\mapsto a_0\cdot f+\sum_{i=1}^n a_i\cdot\frac{\partial f}{\partial x_i}.
\end{equation*} 

\begin{defn}
The \emph{pseudo-Jacobi ideal} $\sJ_{E,\sigma}$ is the image of the map 
\begin{equation}
H^0(Y,E\otimes \Sigma_L\otimes L^\vee)\to H^0(Y,E).
\end{equation} 
\end{defn}

The quotient $H^0(Y,E)/\sJ_{E,\sigma}$ is denoted by $R_{E,\sigma}$.


The $k$-graded piece of the usual Jacobian ideal of a homogeneous polynomial $F$ of degree $d$ is recovered taking $L=\sO_{\mP^n}(d)$ and $E=\sO_{\mP^n}(k)$. In this case it easy to see that $\Sigma_L=\oplus_{i=1}^{n+1}\sO_{\mP^n}(1)$ and sequence (\ref{prolongation}) is the Euler sequence 
\begin{equation}
0\to \sO_{\mP^n}\to\bigoplus^{n+1}\sO_{\mP^n}(1)\to \Theta_{\mP^n}\to0.
\end{equation} The pseudo-Jacobi ideal $\sJ_{\sO_{\mP^n}(k),F}\subset H^0(\mP^n,\sO_{\mP^n}(k))$ is generated by $\frac{\partial F}{\partial x_0},\ldots,\frac{\partial F}{\partial x_n}$, that is it is the degree $k$ part of the Jacobian ideal.

In the case of a smooth algebraic variety $Y$ of dimension $n$ with a smooth hypersurface $X$, we take $L$ to be the sheaf $\sO_Y(X)$, the section $\sigma\in H^0(Y,L)$ is such that $X=\text{div}(\sigma)$, and $E=L=\sO_Y(X)$. We consider the deformations of $X$ inside of the ambient space $Y$. Exactly as in the case of projective hypersurfaces, such an infinitesimal deformation of $X$ is given by $X+tR=0$, $t^2=0$, where $R\in H^0(Y,L)$.
Define $\text{Aut}(Y,L)=\{f\colon Y\to Y\text{ such that } f^*(L)=L\}$. The base of the Kuranishi family for $X$
is $|L|/\text{Aut}(Y,L)$ and we have
\begin{prop}\label{tang}
The tangent space to $|L|/\text{Aut}(Y,L)$ at $X$ is $R_{L,\sigma}$.
\end{prop} \begin{proof} See \cite[Corollary page 48]{green2}.\end{proof}

\section{Infinitesimal Torelli theorem and smooth Grassmannian hypersurfaces}

Following Green's strategy we will reprove the infinitesimal Torelli theorem for hypersurfaces in Grassmannians in a way suitable for later use. A proof valid in a more general context is given in \cite{K}.

Let $\mG=\text{Grass}(s,l+1)$ be the Grassmannian variety of $s$-planes in $\mC^{l+1}$. For $l=1,2$ we obtain only $\mP^1$ and $\mP^2$, hence we will assume $l\geq 3$. Denote by $N=s(l+1-s)$ the dimension of $\mG$ and note that $N\geq l$. Let $X$ be a smooth hypersurface in $|\sO_\mG(a)|$.

To prove that the infinitesimal Torelli holds for $X$ in our case it is enough to show that the map
\begin{equation}
\label{high}
H^1(X,\Theta_X)\to \text{Hom}(H^0(X,K_X),H^1(X,\Omega^{N-2}_X))
\end{equation} is injective. In fact this map is the highest piece of the derivative of the period map, hence if (\ref{high}) is injective, the derivative of the period map is itself injective.

The idea is to prove the surjectivity of the dual of (\ref{high}), which is 
\begin{equation}
H^1(X,\Omega^{N-2}_X)^\vee\otimes H^0(X,K_X)\to H^1(X,\Theta_X)^\vee.
\label{highdual}
\end{equation} 
\begin{lem}
For $a>l$, (\ref{highdual}) fits into the following commutative diagram 
\begin{equation}
\xymatrix{
H^0(X,K_X((N-2)a))\otimes H^0(X,K_X)\ar[r]\ar[d]& H^0(X,K_X^{\otimes2}(N-2)a)\ar@{->>}[d]\\
H^1(X,\Omega^{N-2}_X)^\vee\otimes H^0(X,K_X)\ar[r]&H^1(X,\Theta_X)^\vee
}
\end{equation}
\end{lem}
\begin{proof}
This is, essentially, the content of \cite[Lemma 1.14]{green1}. Following \cite[Lemma 1.10]{green1} the claim follows by the vanishing of certain cohomologies, that we recall here:
\begin{enumerate}
	\item $H^i(\mG,\Omega^j_\mG\otimes K_\mG^{-1}(-m-1)a)=H^{i+1}(\mG,\Omega^j_\mG\otimes K_\mG^{-1}(-m-2)a)=0$ for $0<i<N-1,1\leq j\leq N-1, 1\leq m\leq N-3$
	\item $H^i(\mG,\Omega^j_\mG(-ma))=H^{i+1}(\mG,\Omega^j_\mG(-m-1)a)=0$ for $i<N-1,0\leq j\leq N-1, 1\leq m\leq N-1$.
\end{enumerate}
For (1) take the Serre dual and use the fact that $K_\mG=\sO_\mG(-l-1)$ to obtain
\begin{equation}
h^i(\mG,\Omega^j_\mG\otimes K_\mG^{-1}(-m-1)a)=h^{N-i}(\mG,\Omega^{N-j}(a(m+1)-l-1))
\end{equation} and
\begin{equation}
h^{i+1}(\mG,\Omega^j_\mG\otimes K_\mG^{-1}(-m-2)a)=h^{N-i-1}(\mG,\Omega^{N-j}(a(m+2)-l-1)).
\end{equation} By \cite[Page 171]{snow}, these dimensions are both zero because $N-i>0$, $N-i-1>0$ and $a(m+1)-l-1>l$, $a(m+2)-l-1>l$ for $a>l$.
For (2), in the exact same way we have that 
\begin{equation}
h^i(\mG,\Omega^j_\mG(-ma))=h^{N-i}(\mG,\Omega^{N-j}(am))
\end{equation} and
\begin{equation}
h^{i+1}(\mG,\Omega^j_\mG((-m-1)a))=h^{N-i-1}(\mG,\Omega^{N-j}(a(m+1))).
\end{equation} Again these dimensions are zero because $N-i>0$, $N-i-1>0$ and $am>l$, $a(m+1)>l$ for $a>l$.
Hence for $a>l$ we have the vanishing required in (1) and (2) and we are done.
\end{proof}
\begin{rmk}
Note that in \cite[Lemma 1.14]{green1}, M. Green works in the general case of a smooth sufficiently ample divisor of a smooth variety. Here on the other hand, since we work in the more concrete case of an hypersurface in a Grassmannian manifold, we can give a precise estimate on \lq\lq how ample\rq\rq\ the hypersurface must be. 
\end{rmk}

By the previous lemma we are reduced to study the surjectivity of 
\begin{equation}
H^0(X,K_X((N-2)a))\otimes H^0(X,K_X)\to H^0(X,K_X^{\otimes2}(N-2)a)
\label{highsp}
\end{equation} for $a>l$.
To prove this we go back to the level of $\mG$:
\begin{lem}
If $a>l$ and 
\begin{equation}
H^0(\mG,K_\mG((N-1)a))\otimes H^0(\mG,K_\mG(a))\to H^0(\mG,K_\mG^{\otimes2}(Na))
\end{equation} is surjective, then (\ref{highsp}) is surjective.
\end{lem}
\begin{proof}
From the exact sequence 
\begin{equation}
0\to\sO_\mG(-a)\to \sO_\mG\to\sO_X\to 0
\label{restrizione}
\end{equation} and the adjunction formula $K_X=K_\mG(a)|_X$, we obtain the long exact sequence
\begin{equation}
H^0(\mG,K_\mG^{\otimes2}(Na))\to H^0(X,K_X^{\otimes2}((N-2)a))\to H^1(\mG,K_\mG^{\otimes2}((N-1)a)).
\end{equation}
By the Kodaira vanishing theorem, $H^1(\mG,K_\mG^{\otimes2}(N-1)a))=0$, and by adjuntion we have the following diagram
\begin{equation}
\xymatrix{
H^0(\mG,K_\mG((N-1)a))\otimes H^0(\mG,K_\mG(a))\ar[r]\ar[d]& H^0(\mG,K_\mG^{\otimes2}(N)a)\ar@{->>}[d]\\
H^0(X,K_X((N-2)a))\otimes H^0(X,K_X)\ar[r] &H^0(X,K_X^{\otimes2}(N-2)a).
} 
\end{equation}The thesis immediately follows.
\end{proof}

\begin{prop}\label{etutto}
If $X\subset \mathbb G$ is a smooth divisor $(\sigma=0)$ in $|aH|$ where $a\geq l$ then $H^1(X,\Theta_X)$ is isomorphic to  $R_{\sO_{\mathbb G}(a),\sigma}$.
\end{prop} 
\begin{proof} 
It easy to see using the restriction sequence $0\to \Theta_{\mG}(-a)\to \Theta_\mG\to  \Theta_{\mG|X}\to 0$ and the cohomology vanishings given in \cite[Theorem page 171]{snow} that $H^1(X,\Theta_{\mG|X})=0$.
Hence by the normal exact sequence $0\to \Theta_X\to \Theta_{\mG|X}\to \sO_X(a)\to 0$, we have the identification 
\begin{equation}
H^1(X,\Theta_X)\cong H^0(X,\sO_X(a))/\Ima H^0(X,\Theta_{\mG|X}).
\end{equation} Call $\Sigma$ the sheaf of differential operators introduced in section \ref{sezione1}. The diagram
\begin{equation}
\xymatrix{
0\ar[r]&\sO_\mG\ar[r]\ar[dr]^-\sigma&\Sigma\ar[r]\ar[d]^-{\widetilde{d\sigma}}&\Theta_{\mG}\ar[r]&0\\
&&\sO_{\mG}(a)&&
}
\end{equation} restricted to $X$ gives 
\begin{equation}
\xymatrix{
0\ar[r]&\sO_X\ar[r]\ar[dr]^-0&\Sigma_{|X}\ar[r]\ar[d]^{\widetilde{d\sigma}_{|X}}&\Theta_{\mG|X}\ar[r]\ar[ld]&0\\
&&\sO_{X}(a)&&
}
\end{equation}
and we have that the image of $H^0(X,\Theta_{\mG|X})$ in $H^0(X,\sO_X(a))$ is the same as the image of $H^0(X,\Sigma_{|X})$. Hence 
\begin{equation}
H^1(X,\Theta_X)\cong H^0(X,\sO_X(a))/\Ima H^0(X,\Sigma_{|X})
\end{equation} and it remains to prove that this is isomorphic to $R_{\sO_{\mathbb G}(a),\sigma}=H^0(\mG,\sO_\mG(a))/\Ima H^0(\mG,\Sigma)$. To do this it is enough to check that the kernel of the composition
\begin{equation}
H^0(\mG,\sO_\mG(a))\to H^0(X,\sO_X(a))\to  H^0(X,\sO_X(a))/\Ima H^0(X,\Sigma_{|X})
\end{equation} is exactly $\Ima H^0(\mG,\Sigma)$. This easily follows by the fact that $ H^0(\mG,\Sigma)$ surjects onto $ H^0(X,\Sigma_{|X})$. In fact by the exact sequence $0\to \sO_\mG(-a)\to \Sigma(-a)\to \Theta_\mG(-a)\to 0$ and again the vanishings of \cite[Theorem page 171]{snow} we have that $H^1(\mG, \Sigma(-a))=0$, and we are done.
\end{proof}

\begin{rmk}Proposition \ref{etutto} means that an infinitesimal deformation is trivial if and only if $R$ is an element of the pseudo-Jacobi ideal $\sJ_{L,\sigma}$.
\end{rmk} 

Now we show the infinitesimal Torelli theorem.

\begin{thm}\label{torotoro}
The infinitesimal Torelli theorem holds for smooth hypersurfaces $X$ in $|\sO_\mG(a)|$ if $a>l$.
\end{thm}
\begin{proof}
By the previous lemma it is enough to show that 
\begin{equation}
\label{surG}
H^0(\mG,K_\mG((N-1)a))\otimes H^0(\mG,K_\mG(a))\to H^0(\mG,K_\mG^{\otimes2}(Na))
\end{equation} is surjective. This follows by the fact that the Grassmannian $\mG$ is projectively normal in its Pl\"ucker embedding in $\mP^M$.
This means that the restriction $H^0(\mP^M,\sO(k))\to H^0(\mG,\sO_\mG(k))$ is surjective for $k\geq0$. Hence the surjectivity of (\ref{surG}) follows from the polynomial one at the level of the projective space. Note that the hypothesis $a>l$ ensures that all the twists appearing in (\ref{surG}) are greater than $0$.
\end{proof}

\section{Infinitesimal Torelli and generalized adjoint}

The aim of this section is to study the deformations of a smooth divisor $X$ in $|\sO_\mG(a)|$ with $a>l$ in a Grassmannian variety $\mG$ using the theory of generalized adjoint forms.

\subsection{Twisted normal sequence} 
An infinitesimal deformation $\xi\in \text{Ext}^1(\Omega^1_X,\sO_X)\cong H^1(X,\Theta_X)$ of $X$ gives an exact sequence
\begin{equation}
\label{estensionedivisore}
0\to\sO_X\to \Omega^1_{\sX|X}\to\Omega^1_X\to0. 
\end{equation} By Proposition \ref{etutto} we have that all the deformations of $X$ are inside the ambient space $\mG$, that is, they are given by $R\in \mP(H^0(\mG,\sO_\mG(a)))$. In particular we have that $\xi$ is in the image of the map $H^0(\mG,\sO_\mG(a))\to H^1(X,\Theta_X)$ coming from the normal exact sequence
\begin{equation*}\label{normalll}
0\to \Theta_X\to \Theta_{\mG|X}\to \sO_X(a)\to0
\end{equation*} and the restriction sequence
\begin{equation*}
0\to \sO_\mG\to \sO_{\mG}(a)\to\sO_X(a)\to 0.
\end{equation*}

We can not apply the adjoint theory directly to (\ref{normalll}) because the sheaf $\Omega^1_X$ has no global sections (see point (5) in the lemma below). Hence the idea is to twist sequence  (\ref{normalll}) by a sheaf $\sO_X(k)$. In this way we obtain
\begin{equation}
\label{estensionedivisore1}
0\to\sO_X(k)\to \Omega^1_{\sX|X}(k)\to\Omega^1_X(k)\to0 
\end{equation} which is still associated to the same extension $\xi\in \text{Ext}^1(\Omega^1_X,\sO_X)$ via the isomorphism $$\text{Ext}^1(\Omega^1_X(k),\sO_X(k))\cong \text{Ext}^1(\Omega^1_X,\sO_X).$$ Furthermore if we choose $k$ big enough, we will have that $h^0(X,\Omega^1_X(k))\geq N$, hence we will be able to apply the theory of adjoin forms.

The following lemma proves that $k=2$ is enough for our purposes.
 
\begin{lem}
\label{elenco}
It holds that: 
\begin{enumerate}
	\item $\sO_\mG(a-2)$ is ample;
	\item the cohomology groups $H^i(\mG,\sO_\mG(2))$ vanish for $i\geq1$;
	\item the groups 	$H^i(\mG,\sO_\mG(2-a))$ and $H^i(\mG,\sO_\mG(2-2a))$ vanish for $i<N$;
	\item the groups $H^i(X,\sO_X(2))$ and $H^i(X,\sO_X(2-a))$ vanish for $1\leq i\leq N-2$;
	\item $H^0(X,\Omega^1_X)\cong H^0(X,\Omega^1_{\mG|X})\cong H^0(\mG,\Omega^1_\mG)=0$;
	\item $H^0(X,\Omega^1_X(2))\cong H^0(X,\Omega^1_{\mG|X}(2))\cong H^0(\mG,\Omega^1_\mG(2))$;
	\item $h^0(X,\Omega^1_X(2))\geq N$;
	\item $D_{\Omega^1_X(2)}=0$ i.e. $\Omega^1_X(2)$ is generated by global sections.
\end{enumerate} 
\end{lem}
\begin{proof}
Part (1) is obvious.

For (2) note that $H^i(\mG,\sO_\mG(2))\cong H^i(\mG,K_\mG(l+3))$ since $K_\mG=\sO_{\mG}(-l-1)$. Hence the claim follows by the Kodaira vanishing.

Point (3) also follows by the Kodaira Vanishing theorem.

Part (4) follows by the exact sequence $0\to\sO_\mG(-a)\to \sO_\mG\to\sO_X\to0$ conveniently tensorized and by the vanishing in (2) and (3).

For (5) consider the conormal exact sequence $0\to \sO_X(-a)\to \Omega^1_{\mG|X}\to \Omega^1_X\to 0$. It immediately gives the isomorphism $H^0(X,\Omega^1_{\mG|X})\cong H^0(X,\Omega^1_{X})$. Furthermore the exact sequence $0\to \Omega^1_\mG(-a)\to \Omega^1_{\mG}\to \Omega^1_{\mG|X}\to 0$ and the Nakano vanishing give the isomorphism $H^0(X,\Omega^1_{\mG|X})\cong H^0(X,\Omega^1_{\mG})$. Finally it is well-known that $H^0(X,\Omega^1_{\mG})=0$.

We have that $\sO_\mG(a-2)$ is ample by (1), hence (6) follows in the same fashion as (5).

The dimension of $H^0(\mG,\Omega^1_\mG(2))$ (which is equal to $h^0(X,\Omega^1_X(2))$ by the previous point) is explicitly calculated in \cite[Theorem 3.3]{snow} and it is 
$$\frac{3}{l+2}\binom{l+2}{s+2}\binom{l+2}{s-1}.$$ A simple computation shows that is is grater than $N=s(l+1-s)$.

Finally (8) follows from (6) and the fact that $\Omega^1_\mG(2)$ is generated by global sections (see \cite{snow}).
\end{proof}

\subsection{Infinitesimal deformations and rational forms}
Consider the diagram
\begin{equation}
\label{diagramma8}
\xymatrix{
&&&0&\\
0\ar[r]&\sO_X(2)\ar[r]&\Omega^1_{\sX|X}(2)\ar[r]&\Omega^1_X(2)\ar[u]\ar[r]&0\\
&&&\Omega^1_{\mG|X}(2)\ar[u]&\\
&&&\sO_X(2-a)\ar[u]&\\
&&&0.\ar[u]&}
\end{equation} By the previous Lemma, $H^1(X,\sO_X(2))=H^1(X,\sO_X(2-a))=0$. Hence all the global meromorphic $1$-forms of $\Omega_X^1(2)$ can be lifted both to $H^0(X,\Omega^1_{\sX|X}(2))$ and $H^0(X,\Omega^1_{\mG|X}(2))$.
 
Diagram (\ref{diagramma8}) can be completed as follows
\begin{equation}
\label{diagramma6}
\xymatrix{
&&0&0&\\
0\ar[r]&\sO_X(2)\ar[r]&\Omega^1_{\sX|X}(2)\ar[u]\ar[r]&\Omega^1_X(2)\ar[u]\ar[r]&0\\
0\ar[r]&\sO_X(2)\ar[r]\ar@{=}[u]&\sG\ar[r]\ar[u]&\Omega^1_{\mG|X}(2)\ar[u]\ar[r]&0\\
&&\sO_X(2-a)\ar[u]\ar@{=}[r]&\sO_X(2-a)\ar[u]&\\
&&0\ar[u]&0.\ar[u]&}
\end{equation} By hypothesis our deformation comes from $H^0(\mG,\sO_\mG(a))$, then the horizontal sequence completing diagram (\ref{diagramma6}) is associated to the zero element of $H^1(X,\Theta_{\mG|X})$. Therefore we have the splitting of the second row and a map $\phi$ as follows
\begin{equation}
\label{diagramma7}
\xymatrix{
&&0&0&\\
0\ar[r]&\sO_X(2)\ar[r]&\Omega^1_{\sX|X}(2)\ar[u]\ar[r]&\Omega^1_X(2)\ar[u]\ar[r]&0\\
0\ar[r]&\sO_X(2)\ar[r]\ar@{=}[u]&\sO_X(2)\oplus\Omega^1_{\mG|X}(2)\ar[r]\ar[u]&\Omega^1_{\mG|X}(2)\ar[u]\ar[r]\ar[ul]^{\phi}&0\\
&&\sO_X(2-a)\ar[u]\ar@{=}[r]&\sO_X(2-a)\ar[u]&\\
&&0\ar[u]&0.\ar[u]&}
\end{equation} 
Note that $\det(\Omega^1_{\mG|X}(2))\cong \Omega_X^{N-1}(-a+2N)$ and $\det(\Omega^1_{\sX|X}(2))\cong \Omega^{N-1}_X(2N)$.
\begin{prop}
\label{multip}
The map
\begin{equation*}
\phi^n\colon H^0(X,\Omega_X^{N-1}(-a+2N))\to H^0(X,\Omega_X^{N-1}(2N))
\end{equation*} is given by the section $R_{|X}\in H^0(X,\sO_X(a))$.
\end{prop}
\begin{proof}
This is a local computation. Take on $\mG$ local coordinates $x_1,\ldots,x_{N-1},y$ such that $X$ is given by $y=0$. Then locally the deformation of $X$ is given by $y+tr=0$, where $r$ is a local equation of $R$. From $d(y+tr)=0$ we obtain on $X$ that $dy=-rdt$. Hence if a section of $H^0(X,\det(\Omega^1_{\mG|X}))$ is locally given by $c\cdot dx_1\wedge\ldots\wedge dx_{N-1}\wedge dy$, then its image in $H^0(X,\det(\Omega^1_{\sX|X}))$ is $-rcdx_1\wedge\ldots\wedge dx_{N-1}\wedge dt$. Tensoring by $\sO_{X}(2N)$ gives our thesis.
\end{proof}

Now we construct totally decomposable twisted volume forms. Consider $N$ global sections $\eta_1,\ldots,\eta_N\in H^0(X,\Omega^1_X(2))$ which, by Lemma \ref{elenco} part (7), have unique liftings $\tilde{s}_1,\ldots,\tilde{s}_N\in H^0(X,\Omega^1_{\mG|X}(2))$. Call $\widetilde{\Omega}\in H^0(X,\Omega_X^{N-1}(-a+2N))$ the generalized adjoint form corresponding to $\tilde{s}_1\wedge\dots\wedge \tilde{s}_N$. If we take $s_1:=\phi(\tilde{s}_1),\ldots,s_N:=\phi(\tilde{s}_N)\in H^0(X,\Omega^1_{\sX|X}(2))$, we have that the generalized adjoint $\Omega\in H^0(X,\det (\Omega^1_{\sX|X}(2)))=H^0(X,\Omega_X^{N-1}(2N))$ corresponding to $s_1\wedge\dots\wedge s_N$ is $\Omega=\widetilde{\Omega}\cdot R$. We point out that $\widetilde{\Omega}$ does not depend on the deformation $\xi$, while $\Omega$ obviously does.

\begin{thm}
\label{teoxi}
Assume that $W=\left\langle \eta_1\ldots,\eta_N\right\rangle$ is a generic subspace in $H^0(X,\Omega^1_X(2))$ with $\lambda^NW\neq0$.
Then $R$ is in the pseudo-Jacobi ideal $\sJ_{\sO_\mG(2),s}$ if and only if the adjoint form $\Omega$ is in the image of $H^0(X,\sO_X(2))\otimes \lambda^NW\to H^0(X,\Omega_X^{N-1}(2N))$.
\end{thm}
\begin{proof}
By Lemma \ref{elenco} $(8)$ it follows that $\Omega^1_X(2)$ is generated by its global sections and that $D_{\Omega^1_X(2)}=0$. By \cite[Proposition 3.1.6]{PZ} it follows that $D_W=0$ since $W$ is generic. 
Hence Theorem \ref{theoremA} gives an equivalence. Thus if $\Omega\in \Ima H^0(X,\sO_X(2)\otimes \lambda^NW\to H^0(X,\Omega_X^{N-1}(2N)))$, then 
$\xi\in\ker(H^1(X,\Theta_{X}(2))\to H^1(X,\Theta_{X}(2)\otimes\sO_{X}(D_W)))$, that is $\xi=0$ since $D_W=0$. Viceversa 
 if $R$ is in the pseudo-Jacobi ideal, then the deformation $\xi$ is zero. In particular $\xi\in\ker(H^1(X,\Theta_{X}(2))\to H^1(X,\Theta_{X}(2)\otimes\sO_{X}(D_W)))$ and since $D_W=0$ 
 by the viceversa of Theorem \ref{theoremA} it follows $\Omega\in \Ima H^0(X,\sO_X(2)\otimes \lambda^NW\to H^0(X,\Omega_X^{N-1}(2N)))$.
\end{proof}

\subsection{The Generalized Macaulay's theorem}
We will prove a generalization of Macaulay's theorem which makes explicit in the Grassmannian case the results of \cite[Theorem 2.15]{green1} and of \cite[Theorem page 47]{green2}.
We start with the following lemma which depends on various results contained in \cite{snow}.
\begin{lem}
\label{lemmac}
Call $\Sigma$ the sheaf of differential operators introduced in Section \ref{sezione1}. The cohomology groups $$H^i(\mG,\bigwedge^k\Sigma(-(k-c)a))\quad i\neq0,N$$ vanish for $a\geq l+c+2, c\geq0$. 
\end{lem}
\begin{proof}
Take the exact sequence
\begin{equation}
0\to \sO_\mG\to\Sigma\to \Theta_\mG\to0
\label{sigma}
\end{equation} and its wedge product
\begin{equation}
\label{sigmawedge}
0\to \bigwedge^{k-1}\Theta_\mG\to\bigwedge^{k}\Sigma\to\bigwedge^{k} \Theta_\mG\to0.
\end{equation} We proceed by steps.

\emph{Case k=c.} We have to prove that $H^i(\mG,\bigwedge^k\Sigma)=0$ for $i\neq0,N$. Using sequence (\ref{sigmawedge}) it is enough to prove that $H^i(\mG,\bigwedge^{k-1}\Theta_\mG)=H^i(\mG,\bigwedge^{k}\Theta_\mG)=0$. Note that $H^i(\mG,\bigwedge^{k-1}\Theta_\mG)\cong H^i(\mG,\Omega^{N-k+1}_\mG(l+1))$ and $H^i(\mG,\bigwedge^{k}\Theta_\mG)\cong H^i(\mG,\Omega^{N-k}_\mG(l+1))$ hence the claim follows by \cite[Page 171]{snow}.

\emph{Case $k<c$.} By sequence (\ref{sigmawedge}) twisted by $-(k-c)a$, it is enough to prove that $$H^i(\mG,\bigwedge^{k-1}\Theta_\mG(-(k-c)a))=H^i(\mG,\bigwedge^{k}\Theta_\mG(-(k-c)a))=0.$$ We have the isomorphisms $H^i(\mG,\bigwedge^{k-1}\Theta_\mG(-(k-c)a))\cong H^i(\mG,\Omega_\mG^{N-k+1}(-(k-c)a+l+1))$ and $H^i(\mG,\bigwedge^{k}\Theta_\mG(-(k-c)a))\cong H^i(\mG,\Omega_\mG^{N-k}(-(k-c)a+l+1))$. Since $-(k-c)a+l+1>l$ we conclude again by \cite[Page 171]{snow}.

\emph{Case $k>c$.} Working as it the previous cases we have to prove that 
$$H^i(\mG,\bigwedge^{k-1}\Theta_\mG(-(k-c)a))=H^i(\mG,\bigwedge^{k}\Theta_\mG(-(k-c)a))=0.$$ By Serre duality we will work with the duals
$H^{N-i}(\mG,\Omega_\mG^{k-1}((k-c)a-l-1))$ and $H^{N-i}(\mG,\Omega_\mG^{k}((k-c)a-l-1))$. If $k-c\geq 2$, it immediately follows that $(k-c)a-l-1>l$ and we conclude as in the previous cases. If $k-c=1$, the vanishing of $H^{N-i}(\mG,\Omega_\mG^{k-1}(a-l-1))$ and $H^{N-i}(\mG,\Omega_\mG^{k}(a-l-1))$ follows from \cite[Theorem page 171 (4)]{snow} and the condition $a\geq l+c+2$.
\end{proof}

\begin{rmk}
Note that the condition $a\geq l+c+2$ is really needed only for $k-c=1$. In all the other cases $a>l$ is enough.
\end{rmk}

\begin{thm}[Generalized Macaulay's theorem for Grassmannians]
\label{macaulaygen}

Let $\mG$ and $(\sigma=0)=X\in |aH|$ as above. Then
\begin{enumerate}
	\item $R_{K_\mG^2((N+1)a),\sigma}\cong \mC$;
	\item the standard multiplication map 
	\begin{equation}
	R_{\sO_\mG(ca),\sigma}\otimes R_{K_\mG^2((N+1-c)a),\sigma}\to R_{K_\mG^2((N+1)a),\sigma}\cong \mC.
	\end{equation} 
	is a perfect pairing provided that $a\geq l+c+2$.
\end{enumerate}
\end{thm} 
\begin{proof}
The proof is basically the same as the classical Macaulay's theorem; see \cite[Theorem 6.19]{Vo2}. Using the section $\widetilde{d\sigma}\in H^0(\mG,\Sigma^\vee(a))$ as in equation (\ref{diffs}), we have the Koszul complex: 
\begin{equation}
0\to \bigwedge^{N+1}\Sigma(-(N+1)a)\to\bigwedge^{N}\Sigma(-Na)\to\dots\to \Sigma(-a)\to\sO_\mG\to0. 
\label{koszul}
\end{equation} We tensor (\ref{koszul}) by $\sO_\mG(ca)$: 
\begin{equation}
0\to \bigwedge^{N+1}\Sigma(-(N+1-c)a)\to\bigwedge^{N}\Sigma(-(N-c)a)\to\dots\to \Sigma(-(1-c)a)\to\sO_\mG(ca)\to0. 
\label{koszul2}
\end{equation} Now look at the corresponding hypercohomology spectral sequence. The $E_1^{p,q}$ terms of this spectral sequence are
\begin{equation}
E^{p,q}_1=H^q(\mG,\bigwedge^{N+1-p}\Sigma(-(N+1-p-c)a)).
\end{equation} By Lemma \ref{lemmac}, $E_1^{p,q}=0$ for $q\neq0,N$. Hence it is a spherical spectral sequence with page $E_1$ as follows:

\begin{figure}[h]
\begin{tikzpicture}
	\draw [thick, <->] (0,6) -- (0,0) -- (8,0);
\node [below left] at (0,6) {$q$};
\node [above right] at (8,0) {$p$};
\draw[fill] (0,0) circle [radius=0.07] node[below]{$E_1^{0,0}$};
\draw[fill] (1,0) circle [radius=0.07] node[below]{$E_1^{1,0}$};
\draw[fill] (2,0) circle [radius=0.07] node[below]{$E_1^{2,0}$};
\node [below=0.2] at (3,0) {$\dots$};
\node [below=0.2] at (6,0) {$\dots$};

\draw[fill] (7,0) circle [radius=0.07] node[below]{$E_1^{N+1,0}$};
\draw[fill] (0,5) circle [radius=0.07] node[above right]{$E_1^{0,N}$};
\draw[fill] (1,5) circle [radius=0.07] node[above right]{$E_1^{1,N}$};
\draw[fill] (2,5) circle [radius=0.07] node[above right]{$E_1^{2,N}$};
\node at (3,5) {$\dots$};
\node at (6,5) {$\dots$};

\draw[fill] (7,5) circle [radius=0.07] node[above]{$E_1^{N+1,N}$};
\draw[fill] (0,1) circle [radius=0.07] node[above right]{0};
\draw[fill] (1,1) circle [radius=0.07] node[above right]{0};
\draw[fill] (2,1) circle [radius=0.07] node[above right]{0};
\node at (3,1) {$\dots$};
\node at (6,1) {$\dots$};
\draw[fill] (7,1) circle [radius=0.07] node[above right]{0};

\draw[fill] (7,4) circle [radius=0.07] node[above right]{0};
\draw[fill] (0,4) circle [radius=0.07] node[above right]{0};
\draw[fill] (1,4) circle [radius=0.07] node[above right]{0};
\draw[fill] (2,4) circle [radius=0.07] node[above right]{0};
\node at (3,4) {$\dots$};
\node at (6,4) {$\dots$};
\draw[fill] (7,4) circle [radius=0.07] node[above right]{0};
\node at (1,2) {$\vdots$};
\node at (7,2) {$\vdots$};
\node at (1,3) {$\vdots$};
\node at (7,3) {$\vdots$};
\end{tikzpicture}
\caption{The page $E_1$ of the hypercohomology spectral sequence}
\end{figure}
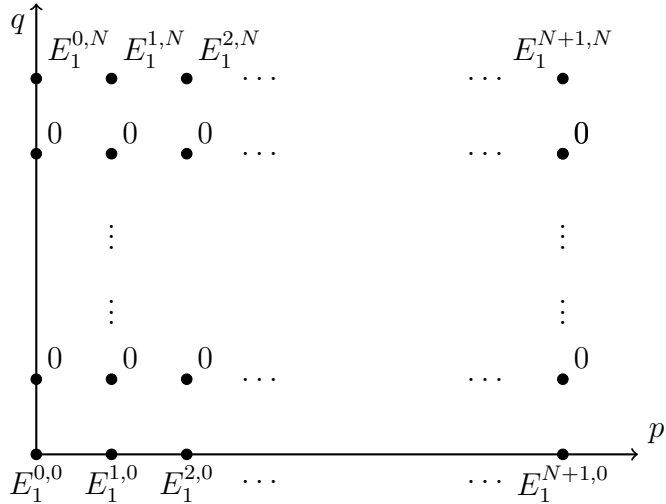
Since the differential is $d_1\colon E_1^{p,q}\to E_1^{p+1,q}$, it follows that 
\begin{equation}
E^{0,N}_2=\Ker (E_1^{0,N}\to E_1^{1,N})=\Ker (H^N(\mG,\bigwedge^{N+1}\Sigma(-(N+1-c)a))\to H^N(\mG,\bigwedge^{N}\Sigma(-(N-c)a))),
\end{equation} that is, using Serre duality and the isomorphism $\bigwedge^{N+1}\Sigma\cong \bigwedge^{N}\Theta$
\begin{equation}
E^{0,N}_2=(H^0(\mG,K_\mG^2((N+1-c)a))/\Ima H^0(\mG,\Sigma\otimes K_\mG^2((N-c)a)))^\vee=R_{K_\mG^2((N+1-c)a),\sigma}^\vee.
\end{equation} Furthermore
\begin{equation}
E^{N+1,0}_2=E_1^{N+1,0}/\Ima E_1^{N,0}=H^0(\mG,\sO_\mG(ca))/\Ima H^0(\mG,\Sigma(-(1-c)a))=R_{\sO_\mG(ca),\sigma}.
\end{equation}
The fact that the spectral sequence is spherical gives that $E^{0,N}_2=E^{0,N}_3=\dots=E^{0,N}_{N+1}$ and $E^{N+1,0}_2=E^{N+1,0}_3=\dots=E^{N+1,0}_{N+1}$. Moreover this spectral sequence abuts to the hypercohomology of (\ref{koszul2}) which is zero because (\ref{koszul2}) is exact. Hence $d_{N+1}$ is an isomorphism $d_{N+1}\colon E^{0,N}_{N+1}\to E^{N+1,0}_{N+1}$, that is 
\begin{equation}
R_{\sO_\mG(ca),\sigma}\cong R_{K_\mG^2((N+1-c)a),\sigma}^\vee.
\label{seqspe}
\end{equation}
Choosing $c=0$ gives part (1) of our thesis. Furthermore since the multiplication with $H^0(\mG,K_\mG^2((N+1-c)a))$ gives a map of the entire spectral sequence, we deduce that (\ref{seqspe}) is compatible with the multiplication in the sense that given $\nu\in H^0(\mG,K_\mG^2((N+1-c)a))$ we have a commutative diagram
\begin{equation}
\xymatrix{
R_{\sO_\mG(ca),\sigma}\ar[r]^-\cong\ar[d]^\nu& (R_{K_\mG^2((N+1-c)a),\sigma})^\vee\ar[d]^{\nu^\vee}\\
R_{K_\mG^2((N+1)a),\sigma}\ar[r]^-\cong&(R_{\sO_\mG,\sigma})^\vee
}
\end{equation} given by the $\nu$ and its dual. It follows that the pairing induced by (\ref{seqspe})
is given by multiplication.
\end{proof}

\begin{rmk}We can give a better lower bound for $a$ in the Generalized Macaulay's theorem \ref{macaulaygen} assuming $c=1$. In this case we have to consider a non-spherical spectral sequence. 
In order to do that, we leave apart only the  classically well-known case where $\mG$ is the standard projective space $\mP^l$. \end{rmk}

\begin{thm} [Macaulay's theorem for Grassmannians]
\label{macaulaygenc}
Let $\mG=\text{Grass}(s,l+1)$ with $s\neq1$ and $X$ as above. Then
\begin{enumerate}
	\item $R_{K_\mG^2((N+1)a),\sigma}\cong \mC$
	\item the multiplication map 
	\begin{equation}
	R_{\sO_\mG(a),\sigma}\otimes R_{K_\mG^2(Na),\sigma}\to R_{K_\mG^2((N+1)a),\sigma}\cong \mC.
	\end{equation} 
	is a perfect pairing provided that $a>l$.
\end{enumerate}
\end{thm} 
\begin{proof}
If $a\geq l+c+2=l+3$, we can apply Theorem \ref{macaulaygen}. Hence we will assume $l<a<l+3$. The $E_1$ terms of the hypercohomology spectral sequence are 
$$E^{p,q}_1=H^q(\mG,\bigwedge^{N+1-p}\Sigma(-(N-p)a)).
$$ Note that the hypothesis $a\geq l+3$ of Lemma \ref{lemmac} is really needed only in the case of $k-c=1$, that is $N-p=1$, hence $E_1^{p,q}=0$ for $q\neq 0,N$ and $p\neq N-1$.  
The $E_1$ page is then 
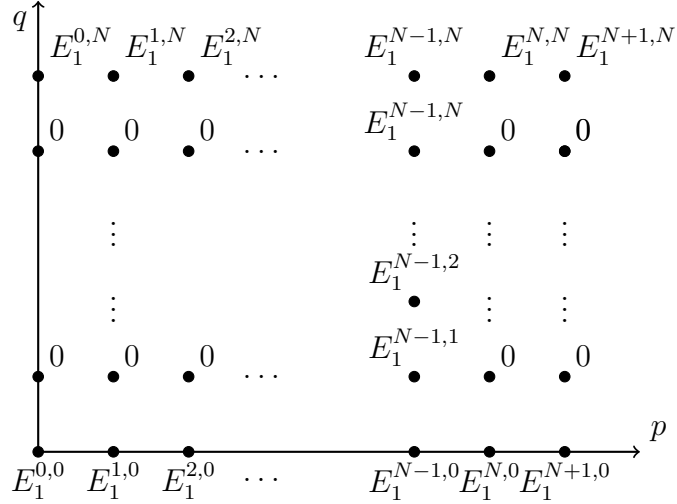
\begin{center}
\begin{tikzpicture}
	\draw [thick, <->] (0,6) -- (0,0) -- (8,0);
\node [below left] at (0,6) {$q$};
\node [above right] at (8,0) {$p$};
\draw[fill] (0,0) circle [radius=0.07] node[below]{$E_1^{0,0}$};
\draw[fill] (1,0) circle [radius=0.07] node[below]{$E_1^{1,0}$};
\draw[fill] (2,0) circle [radius=0.07] node[below]{$E_1^{2,0}$};
\node [below=0.2] at (3,0) {$\dots$};

\draw[fill] (7,0) circle [radius=0.07] node[below]{$E_1^{N+1,0}$};
\draw[fill] (0,5) circle [radius=0.07] node[above right]{$E_1^{0,N}$};
\draw[fill] (1,5) circle [radius=0.07] node[above right]{$E_1^{1,N}$};
\draw[fill] (2,5) circle [radius=0.07] node[above right]{$E_1^{2,N}$};
\node at (3,5) {$\dots$};

\draw[fill] (5,5) circle [radius=0.07] node[above ]{$E_1^{N-1,N}$};
\draw[fill] (6,5) circle [radius=0.07] node[above right]{$E_1^{N,N}$};
\draw[fill] (7,5) circle [radius=0.07] node[above right]{$E_1^{N+1,N}$};
\draw[fill] (6,1) circle [radius=0.07] node[above right]{0};
\draw[fill] (6,4) circle [radius=0.07] node[above right]{0};
\draw[fill] (0,1) circle [radius=0.07] node[above right]{0};
\draw[fill] (1,1) circle [radius=0.07] node[above right]{0};
\draw[fill] (2,1) circle [radius=0.07] node[above right]{0};
\node at (3,1) {$\dots$};
\draw[fill] (7,1) circle [radius=0.07] node[above right]{0};

\draw[fill] (7,4) circle [radius=0.07] node[above right]{0};
\draw[fill] (0,4) circle [radius=0.07] node[above right]{0};
\draw[fill] (1,4) circle [radius=0.07] node[above right]{0};
\draw[fill] (2,4) circle [radius=0.07] node[above right]{0};
\node at (3,4) {$\dots$};
\draw[fill] (7,4) circle [radius=0.07] node[above right]{0};
\node at (1,2) {$\vdots$};
\node at (7,2) {$\vdots$};
\node at (1,3) {$\vdots$};
\node at (7,3) {$\vdots$};
\node at (6,3) {$\vdots$};
\node at (6,2) {$\vdots$};
\node at (5,3) {$\vdots$};
\draw[fill] (5,0) circle [radius=0.07] node[below]{$E_1^{N-1,0}$};
\draw[fill] (5,1) circle [radius=0.07] node[above]{$E_1^{N-1,1}$};
\draw[fill] (5,4) circle [radius=0.07] node[above]{$E_1^{N-1,N}$};
\draw[fill] (6,0) circle [radius=0.07] node[below]{$E_1^{N,0}$};
\draw[fill] (5,2) circle [radius=0.07] node[above]{$E_1^{N-1,2}$};

\end{tikzpicture}
\end{center} 
\vspace{0.5cm}
To make sure that $E^{0,N}_2=E^{0,N}_3=\dots=E^{0,N}_{N+1}$ and $E^{N+1,0}_2=E^{N+1,0}_3=\dots=E^{N+1,0}_{N+1}$ as in the proof of Theorem \ref{macaulaygen}, it is enough to prove that $E_1^{N-1,1}=E_1^{N-1,2}=0$. Now we have that $E_1^{N-1,1}=H^1(\mG,\bigwedge^{2}\Sigma(-a))$ and $E_1^{N-1,2}=H^2(\mG,\bigwedge^{2}\Sigma(-a))$. Using the exact sequence
$$0\to \Theta(-a)\to \bigwedge^2\Sigma(-a)\to\bigwedge^2\Theta(-a)\to 0 $$ it is enough to show the vanishing of 
$H^1(\mG,\Theta(-a))$, $H^2(\mG,\Theta(-a))$, $H^1(\mG,\bigwedge^2\Theta(-a))$ and $H^2(\mG,\bigwedge^2\Theta(-a))$. 

If $a=l+1$, then $\sO_\mG(-a)=K_\mG$ and these vanishing are classically known.

If $a=l+2$, we take the Serre dual and obtain $H^{N-1}(\mG,\Omega^1(1))$, $H^{N-2}(\mG,\Omega^1(1))$, $H^{N-1}(\mG,\Omega^2(1))$ and $H^{N-2}(\mG,\Omega^2(1))$. The vanishing of these groups come from \cite[Theorem page 171 part (1)]{snow} using the fact that under our hypotheses $N\geq4$ and $s-1>0$.

Hence $d_{N+1}\colon E^{0,N}_2\to E^{N+1,0}_2$  gives an isomorphism and the proof proceeds like in Theorem \ref{macaulaygen}.

\end{proof}

\subsection{Twisted one forms and twisted decomposable volume forms}
Twisted forms on Grassmannians provide us a natural setting to apply the Generalized Adjoint Theory. Indeed by \cite{BW}, $H^0 (\mG, {\rm{det}} ( \Omega^{1}_{\mG}(m))) $ is an irreducible representation and since, for every $m\geq 2$, 
$\Omega^{1}(m)$ is globally generated, we have easily that any element of $H^0 (\mG, {\rm{det}} ( \Omega^{1}_{\mG}(m)))$ is actually obtainable as a $\mathbb C$-linear combination of totally decomposable forms. More precisely we have, in the case $m=2$:
\begin{prop}
\label{surgettivita}
The natural map  induced by the wedge product 
\begin{equation}
\bigwedge^{N}H^0(X,\Omega^1_{\mG|X}(2))\to H^0(X, \det (\Omega^1_{\mG|X}(2)))
\end{equation}is surjective.
\end{prop}
\begin{proof}
By Borel-Weil theorem, see also cf: \cite[Proposition 10.2]{Bo}, we know that $H^0(\mG, \det (\Omega^1_{\mG}(2)))$ is an irreducible representation. By Lemma \ref{elenco} $(8)$ we have that $\Omega^1_\mG(2)$ is globally generated by its global sections. Now the natural homomorphism
\begin{equation}
\bigwedge^{N}H^0(\mG,\Omega^1_{\mG}(2))\to H^0(\mG, \det (\Omega^1_{\mG}(2)))
\end{equation} is surjective by Schur's Lemma. By Lemma \ref{elenco} we know that $H^0(\mG,\Omega^1_{\mG}(2))\cong H^0(X,\Omega^1_{\mG|X}(2))$. Hence the claim follows if we show that $H^0(\mG, \det (\Omega^1_{\mG}(2)))\to H^0(X, \det (\Omega^1_{\mG|X}(2)))$ is surjective. Indeed this follows by the exact sequence
\begin{equation}
0\to \Omega^N_\mG(2N-a)\to \Omega^N_\mG(2N)\to \Omega^N_\mG(2N)_{|X}\to0
\end{equation} and the Kodaira vanishing applied to $H^1(\mG,\Omega^N_\mG(2N-a))$.
\end{proof}

\subsection{Volume forms and the infinitesimal Torelli theorem}
We link the global forms of $\Omega_\mG^N(2N)$, which are objects coming from the ambient variety $\mG$, to the infinitesimal deformations of $X\subset \mG$ contained in pseudo-Jacobi ideal $\sJ_{\sO_\mG(X),\sigma}$.
\begin{lem}\label{quattro}
Let $\mG=\text{Grass}(s,l+1)$ with $s\neq1$ and $X$ as above. Then the infinitesimal deformation $R$ is in the pseudo-Jacobi ideal $\sJ_{\sO_\mG(X),\sigma}$ if and only if $R\widetilde{\Omega}\in \sJ_{\Omega_\mG^N(2N+a),\sigma}$ for every section $\widetilde{\Omega}\in H^0(\mG,\Omega_\mG^N(2N))$ which restricts to a generalized adjoint relative to the vertical exact sequence of diagram (\ref{diagramma8}).
\end{lem}

\begin{proof}
Note that a generalized adjoint relative to the vertical sequence of diagram (\ref{diagramma8}) is in fact an element of $$H^0(X,(\Omega_\mG^N(2N))_{|X})=H^0(X,\Omega_X^{N-1}(2N-a)).$$
We want to apply the generalized version of Macaulay's theorem \ref{macaulaygenc}. We only know that $R\widetilde{\Omega}\in \sJ_{\Omega_\mG^N(2N+a),\sigma}$ for $\widetilde{\Omega}\in H^0(\mG,\Omega_\mG^N(2N))$ which restricts to a generalized adjoint. So now we prove that this is enough to have that 
\begin{equation*}
R\cdot H^0(\mG,\Omega_\mG^N(2N))\subset \sJ_{\Omega_\mG^N(2N+a),\sigma}.
\end{equation*}
Consider the restriction sequence
\begin{equation*}
0\to \Omega_\mG^N(2N-a)\to\Omega_\mG^N(2N)\to\Omega_\mG^N(2N)_{|X}\to 0.
\end{equation*} 
As we have seen in the proof of Proposition \ref{surgettivita},  
\begin{equation*}
0\to H^0(\mG,\Omega_\mG^N(2N-a))\to H^0(\mG,\Omega_\mG^N(2N))\to H^0(X,\Omega_\mG^N(2N)_{|X})\to0
\end{equation*} is exact. 
By Proposition \ref{surgettivita} we can also assume that all the global sections of $H^0(X,\Omega_\mG^N(2N)_{|X})$ are in fact linear combinations of generalized adjoints. Hence our hypothesis that $R\widetilde{\Omega}\in \sJ_{\Omega_\mG^N(2N-a),\sigma}$ for every section $\widetilde{\Omega}\in H^0(\mG,\Omega_\mG^N(2N))$ which restricts to a generalized adjoint in $H^0(X,\Omega_\mG^N(2N)_{|X})$, together with the fact that the map $$H^0(\mG,\Omega_\mG^N(2N-a))\to H^0(\mG,\Omega_\mG^N(2N)))$$ is given by the multiplication by $s$, which is an element of the pseudo-Jacobi ideal, implies that 
\begin{equation*}
R\cdot H^0(\mG,\Omega_\mG^N(2N))\subset \sJ_{\Omega_\mG^N(2N+a),\sigma}.
\end{equation*}

Now we apply Macaulay's theorem \ref{macaulaygenc} to deduce that $R$ is in the pseudo-Jacobi ideal. It is enough to show that
\begin{equation*}
R_{\Omega_\mG^N(Na-2N),\sigma}\otimes R_{\Omega_\mG^N(2N),\sigma}\to R_{(\Omega_\mG^{2N})(Na),\sigma}
\end{equation*} is surjective. This follows from the surjectivity at the level of the $H^0$:
\begin{equation*}
H^0(\mG,\Omega_\mG^N(Na-2N))\otimes H^0(\mG,\Omega_\mG^N(2N))\to H^0(\mG,(\Omega_\mG^{2N})(Na))
\end{equation*} which holds by projective normality.
\end{proof}

\subsection{Proof of the Main Theorem}

 In this proof we denote as always by $\xi\in H^1(X,\Theta_X)$ the infinitesimal deformation and by $[R]\in R_{\sO_{\mathbb G}(a),\sigma}$ the corresponding element given by Proposition \ref{etutto}. We will use also Remark \ref{zerozero}. 

$i)\Leftrightarrow ii)$. If $d\sP(\xi)=0$ then, by Theorem \ref{torotoro}, $\xi=0$. By Proposition \ref{etutto} this means $[R]=0$, that is $R\in \sJ_{\sO_\mG(a),\sigma}$. The viceversa is trivial by Proposition \ref{etutto}.  
%

$ii) \Leftrightarrow iii)$ This is Theorem \ref{teoxi}.

 $ii)\Leftrightarrow iv)$ This is the content of Lemma \ref{quattro}. 
\medskip

\noindent
{\bf{ Acknoledgments}} The authors want to thank Giorgio Ottaviani for his advices on the use of the Borel-Weil theorem.

\end{document}